\documentclass[12pt,reqno] {amsart}
\usepackage{graphicx}
\usepackage{tikz}
\usepackage[utf8]{inputenc}
\usepackage{amsmath, amssymb, amsfonts}
\usepackage{mathrsfs}
\usepackage{xcolor}
\usepackage{todonotes}
\usepackage{ulem}

\usepackage{hyperref}
\usepackage{mathtools}
\usepackage{amsmath}
\usepackage{amssymb}
\usepackage{amsthm}
\newtheorem{proposition}{Proposition}
\newtheorem{question}{Question}
\newtheorem{lemma}{Lemma}
\newtheorem{theorem}{Theorem}
\newtheorem{cor}{Corollary}
\theoremstyle{remark}
\newtheorem{remark}[theorem]{Remark}

\theoremstyle{definition}
\newtheorem{definition}[theorem]{Definition}
\newtheorem{example}[theorem]{Example}

\title[von Neumann Inequality and Doubly Contractive Weighted Shift]{von Neumann Inequality for a class of Doubly Contractive Weighted Shifts}
 \author[S. Dey]{Soumyadip Dey}
 \author[R. Gupta]{Rajeev Gupta}
\author[S. Kumar]{Surjit Kumar}
 \address[S. Dey]{School of Mathematics and Computer Science\\Indian Institute of Technology Goa, India}
 \email{soumyadip22232101@iitgoa.ac.in}
 \address[R. Gupta]{School of Mathematics and Computer Science\\Indian Institute of Technology Goa, India}
 \email{rajeev@iitgoa.ac.in}
 \address[S. Kumar]{Department of Mathematics \\  Indian Institute of Technology Madras, Chennai - 600036, India}
 \email{surjit@iitm.ac.in}

   \subjclass[2020]{Primary 47B37, 47A20 Secondary 47A13}
\keywords{multivariable weighted shift, von Neumann's inequality, spherical dilation, doubly contractive, balanced weighted shift}
                             
\begin{document}

\begin{abstract}
In this article, we investigate the ball version of von Neumann inequality for the class of doubly contractive $d$-tuple of weighted shift. We show that if the weighted shift is balanced or satisfies an appropriate weight condition, then it admits a spherical unitary dilation. Consequently, such tuples satisfy the von Neumann inequality over Euclidean unit ball. For the general class of commuting tuple of doubly contractive operators (not necessarily weighted shift) on a Hilbert space, we further establish von Neumann inequality for homogeneous polynomials of degree at most $2.$ 
\end{abstract}

\maketitle

\section{Introduction}
The von Neumann inequality \cite{vN} states that $\|p(T)\|\leqslant \|p\|_{\mathbb D,\infty}$
for all polynomials $p\in\mathbb{C}[z]$ and for any contraction $T$ on a complex Hilbert space. 
Here $\|p\|_{\Omega,\infty}$ denotes the supremum of $|p|$ over the domain $\Omega.$ Whenever $\Omega$ is clear from the context, we shall omit it from the notation $\|p\|_{\Omega, \infty}.$
A proof of this non-trivial inequality can be obtained through Sz.-Nagy dilation theorem \cite{Sz-dilation} which states that for any contraction $T$ on a separable Hilbert space $\mathcal H$ there is a unitary dilation $U$ (on a possibly bigger Hilbert space $\mathcal K$) associated with $T$. This, in turn, implies that $p(T)$ is obtained by compressing $p(U)$ to the Hilbert space $\mathcal H.$

In what follows, $\mathbb N$ and $\mathbb Z_+$ denote the set of all positive integers and the set of all non-negative integers respectively. Two natural higher dimensional analogues to the unit disc $\mathbb D$ are the Euclidean unit ball $\mathbb B_d$ and the unit polydisc $\mathbb D^d.$  Since the emergence of the article (1951)\cite{vN}, the generalization of von Neumann inequality to these balls have been active area of research leading to many interesting results. Since the class of weighted shifts are concrete and tractable, it becomes important in providing examples, counterexamples, and models across many areas. The validity of von Neumann inequality to the class of weighted shift has been an interesting question since (1974)\cite{Shields}. 

For any pair of commuting contractions $(T_1,T_2)$, a generalization of the von Neumann inequality:
$\|p(T_1,T_2)\|\leqslant \|p\|_{\mathbb D^2,\infty},$ $p\in\mathbb{C}[z_1,z_2],$ 
follows from a deep theorem of And$\hat{\mbox{o}}$ (1963)\cite{Ando} on unitary dilation of a pair of commuting contractions. For $d\in\mathbb N,$ let 
$\mathscr{C}_d$ denote the set of all $d$-tuples $\boldsymbol T=(T_1,\ldots,T_d)$ of commuting contractions 
on some Hilbert space $\mathcal{H}.$ 
Parrott (1970)\cite{Parrott} produced an element in $\mathscr C_3$ which satisfies von Neumann’s inequality but does not dilate to a tuple of commuting
unitaries. This proves that getting commuting unitary dilation of a tuple of commuting contraction may not be guaranteed even if the commuting tuple of contractions satisfies von Neumann inequality. The question of validity of von Neumann inequality in multi-variable was answered in the negative by Varopoulos in the paper \cite{V1}, where he showed 
that the von Neumann inequality 
fails for $\boldsymbol T$ in $\mathscr{C}_d,$ for some $d > 2.$ 
In the addendum of the same article, he along with Kaijser \cite{V1} and 
simultaneously Crabb and Davie \cite{CD} 
produced an explicit example of 
three commuting contractions $T_1,T_2,T_3$ 
and a polynomial $p$ for which 
$\|p(T_1,T_2,T_3)\| > \|p\|_{\mathbb D^3,\infty}.$ 
Since then it has been one of the peculiar topics in operator theory. Even though dilation method to prove von Neumann inequality doesn't work in general, it may still work if we restrict our attention to some specific class of commuting contractions, for instance, Hartz (2017)\cite{Hz} answered the question of dilation of commuting contractive weighted shifts affirmatively and proved the following result:
\begin{theorem}\cite[Theorem 1.1]{Hz}
Let $\boldsymbol T = (T_1, \ldots, T_d)$ be a contractive $d$-variable weighted shift with non-zero weights. Then $\boldsymbol T$ dilates to a $d$-tuple of commuting unitaries.
\end{theorem}

A ball analogue of the von Neumann inequality can have various formulations. For instance, one might consider the von Neumann inequality for the class of row contractions or for the class of spherical contractions. However, for both of these classes, the von Neumann inequality does not hold, as evidenced by the works of  Drury (1978)\cite{Drury1978} and Hartz (2022)\cite{Hartz2022}. 
In the view of this, the following question seems very natural in the realm of weighted shifts.
\begin{question}\label{On Ball}
    Suppose $\boldsymbol T:=(T_1,\ldots, T_d)$ is a row contraction and spherical contraction on a Hilbert space. Further, assume that $\boldsymbol T$ is a $d$-variable weighted shift, does it follow that $\|p(\boldsymbol T)\|\leqslant \|p\|_{\mathbb B_d,\infty}?$
\end{question}
In fact, more generally one can ask if a $d$-variable weighted shift which is both row contraction and spherical contraction, can it be dilated to spherical isometry? In \cite{M-V}, M\"uller and Vasilescu proved that every commuting $d$-tuple operators which is $d$-hypercontraction dilates to a spherical isometry (see also \cite{At-1}). 
Dilation is much stronger notion than that of von Neumann inequality.  
In Example \ref{Counter-example: column} and Example \ref{Counter-example: row}, we see that if a weighted shift is row contractive or spherical contractive, it may not satisfy von Neumann inequality. 
In fact, the following remark highlights the necessity of considering both row contractions and spherical contractions for a weighted shift to admit a spherical unitary dilation.
 If a commuting $d$-tuple $\boldsymbol T=(T_1,\ldots, T_d)$ of bounded linear operators admits a spherical dilation then a routine verification shows that $\boldsymbol T$ must be both spherical and row contraction. 
 
\begin{definition}[Doubly Contractive Operator tuple]
Let \( \boldsymbol T = (T_1, \dots, T_d) \) be a commuting \( d \)-tuple of operators on \( \mathcal{B}(\mathcal{H}) \), the bounded linear operators on a Hilbert space \( \mathcal{H} \). Then \( \boldsymbol T \) is said to be a \textit{doubly contractive operator tuple} if \( \sum_{j=1}^{d} T_j^* T_j \leq I \), and \( \sum_{j=1}^{d} T_j T_j^* \leq I \).
\end{definition}

\begin{definition}[Spherical Isometry]
Let \( \boldsymbol T = (T_1, \dots, T_d) \) be a commuting \( d \)-tuple of operators on \( \mathcal{B}(\mathcal{H}) \). Then \( \boldsymbol T \) is called a \textit{spherical isometry} if \( \sum_{j=1}^{d} T_j^* T_j = I \). 
\end{definition}

\begin{definition}[Spherical Dilation]
    A commuting $d$-tuple $U=(U_1, \ldots, U_d)$ is said to be spherical unitary if each $U_j$ is normal and $\sum_{j=1}^d U^*_jU_j = I.$ A commuting $d$-tuple $T=(T_1, \ldots, T_d)$ in $\mathcal B(\mathcal H)$ is said to admit a {\it spherical dilation} if there exist a Hilbert space $\mathcal K \supset \mathcal H$ and a  spherical unitary $U=(U_1, \ldots, U_d)$ in $\mathcal B(\mathcal K)$
 such that $T^I=P_{\mathcal H}U^I|\mathcal H$ for all $I \in \mathbb Z_{+}^d,$ where $P_{\mathcal H}$ is the orthogonal projection of $\mathcal K$ onto $\mathcal H$ and $\mathbb Z_+$ denotes the set of all non-negative integers. 
\end{definition}

The main aim of this article is an attempt to tackle the problem stated in Question \ref{On Ball} by developing a technique parallel to that in \cite{Hz}. On the way, we get the following positive result.

\begin{theorem}\label{Main Theorem}
Let $\boldsymbol T = (T_1, \ldots, T_d)$ be a doubly contractive $d$-variable weighted shift with weights $\{w_{I,j}: I\in \mathbb Z_+^d,\ 1\leq j\leq d \}.$ Suppose there are constants $m_1,\ldots,m_d\in \mathbb R$ such that for each $j\in \{1,\dots ,d\},$ $|w_{I,j}|\leqslant m_j$ for all $I\in\mathbb Z_+^d$ and \(\sum_{j=1}^d m_j^2 \leqslant 1\). Then $\boldsymbol T$ admits a spherical dilation. In particular, $\boldsymbol T$ satisfies von Neumann’s inequality.
\end{theorem}
In Theorem \ref{thm-dila}, we show that every spherical contractive $d$-tuple of operators of multiplication by the co-ordinate functions on a spherically balanced Hilbert space of formal power series (see Definition \ref{Spherically Balanced- definition}) admits a spherical dilation. Even though in Example \ref{Counter-example: column} and in Example \ref{Counter-example: row}, a homogeneous polynomial of degree $2$ suffices to create counter-examples for concerned formulation of the von Neumann inequality, in Proposition \ref{von Neumann's inequality for doubly contractive operators}, we show that a commuting tuple of doubly contractive operators on a Hilbert space always satisfies von Neumann inequality for homogeneous polynomials of degree at most $2.$  

\section{Weighted Shift}
Let \( d \in \mathbb{N} \).  For an element \( I \in \mathbb{Z}_+^d \) and \( 1 \leq j \leq d \), we write \( \epsilon_j \) for the multi-index \( I = (i_1, \dots, i_d) \) with \( i_j = 1 \) and \( i_k = 0 \) for \( k \neq j \). Given a multi-index \( I = (i_1, \dots, i_d) \), we define \( |I| = i_1 + i_2 + \dots + i_d \). 
Given two multi-indices \( I = (i_1, \dots, i_d) \) and \( J = (j_1, \dots, j_d) \), we say that \( I \leq J \) if \( i_k \leq j_k \) for all \( 1 \leq k \leq d \). 

Let $({w}_{I,j})_{(I,j) \in \mathbb{Z}_+^d \times \{1, \dots, d\}}$
be a bounded collection of complex numbers satisfying the commutation relations
\begin{eqnarray}\label{commutation relation}
    {w}_{I,j} \; {w}_{I+\epsilon_j ,k} = {w}_{I,k} \; {w}_{I+\epsilon_k ,j} \quad \text{for all } I \in \mathbb{Z}_+^d \text{ and } j,k \in \{1,\dots ,d\}.
\end{eqnarray}
Let \( \mathcal{H} \) be a Hilbert space with an orthonormal basis \( \{ e_I : I \in \mathbb{Z}_+^d \} \). A {\it $d$-variable weighted shift} with weights $({w}_{I,j})$ is the unique $d$-tuple of bounded linear operators \( (T_1 ,\dots ,T_d ) \) on \( \mathcal{H} \) satisfying
\begin{eqnarray*}
    T_j e_I = {w}_{I,j} e_{I+\epsilon_j}, \quad \text{for } I \in \mathbb{Z}_+^d, \, j \in \{1,\dots,d\}.
\end{eqnarray*}
Note that the adjoint of $T_j$ can be computed to be
\[
T_j^* e_I =
\begin{cases}
\overline{{w}}_{I-\epsilon_j,j} e_{I-\epsilon_j}, & \text{if } i_j \geq 1, j\in \{1,\dots,d\} \\
0, & \text{if } i_j = 0.
\end{cases}
\]
Observe that the relations \eqref{commutation relation} guarantee that the tuple of  operators \( (T_1,\ldots, T_d) \) is a commuting tuple.  Evidently, \( T_j \) is a contraction if and only if \( |{w}_{I,j}| \leq 1 \) for all \( I \in \mathbb{Z}_+^d \). We refer the reader to the \cite{J-L} for further study of weighted shifts operators. The following lemma is taken from the same article, see \cite[Corollary 2]{J-L}.
\begin{lemma}  
Let \( T \) be a \( d \)-variable  weighted shift on \( {\ell^2}(\mathbb{Z}_+^d) \) with non-zero weights \( (w_{I,j}) \). Then \( T \) is unitarily equivalent to the \( d \)-variable weighted shift with weights \( (|w_{I,j}|) \).
\end{lemma} 

For \( N \in \mathbb{N} \), define the finite-dimensional subspace
\[
\mathcal{H}_N = \operatorname{span} \{ e_I : |I| \leq N \}.
\]
Suppose \( (w_{I,j}) \) is a collection of complex numbers satisfying the commutation relations \eqref{commutation relation} for \( |I| \leq N - 1 \) and \( j \in \{1, \dots, d\} \). 
The truncated ($d$-variable) weighted shift with weights \( (w_{I,j}) \) is the unique $d$-tuple of operators \( (T_1, \dots, T_d) \) on \( \mathcal{H}_{N+1} \) satisfying
\[
T_j e_I =
\begin{cases}
{w}_{I,j} e_{I+\epsilon_j}, & \text{if } |I| \leq N, \\
0, & \text{if } |I| = N+1.
\end{cases}
\]

It is easy to see that a \( d \)-variable weighted shift \( (T_1, \dots, T_d) \) with weights \( (w_{I,j}) \), for \( (I,j) \in \mathbb{Z}_+^d \times \{1,\ldots,d\} \), is doubly contractive if and only if 
\[
\sum_{j=1}^d |{w}_{I,j}|^2 \leq 1, \quad
\sum_{j=1}^d \big| {w}_{I - \epsilon_j ,j} \big|^2 \leq 1,
\]
with the understanding that for all $k\in\{1,\ldots, d\},$ $w_{J,k}=0$ whenever $J\notin \mathbb Z_{+}^d.$
Let $\{\beta_I\}_{I \in \mathbb{Z}_+^d}$ be a multi-sequence of
positive numbers such that $\beta _0 =1$ and
\begin{eqnarray*}
    \sup \left\{\frac{\beta_{I+\varepsilon_j}} {\beta_I} : 1 \leq j \leq d, \;  I \in \mathbb{Z}_+^d\right\} < \infty. 
\end{eqnarray*}
Consider the Hilbert space $H^2(\beta)$ of formal power series
$f(z)=\sum_{I \in  \mathbb{Z}_+^d} f_I z^I$ such that
$$ \|f\|^2_{H^2(\beta)}=\sum_{\alpha \in  \mathbb{Z}_+^d}|f_I |^2 \beta^2_I <
\infty.$$
Every $d$-variable weighted shift $T$ is unitarily equivalent to the $d$-tuple $M_z=(M_{z_1}, \ldots, M_{z_d})$ of multiplication by the co-ordinate functions $z_1, \ldots, z_d$ on $H^2(\beta),$ where $\beta_I = \|T^Ie_0\|$  for all $I \in \mathbb{Z}_+^d$ (refer to  \cite[Proposition 8]{J-L}). 
The weight $w_{I,j}$ and $\beta_I$ is related by
\begin{eqnarray*}
     w_{I,j} =\frac{\beta_{I+\varepsilon_j}} {\beta_I}  \quad \text{for all } I \in \mathbb{Z}_+^d \text{ and } j \in \{1,\dots ,d\}.
\end{eqnarray*}
For a fixed $z \in \mathbb C^d, $ consider the {\it slice} $f_z$ of a formal power series $f(z)=\sum_{I \in  \mathbb{Z}_+^d} f_I z^I$ at $t \in \mathbb C$ given by
\[f_z(t)=f(tz_1, \ldots, tz_d)=\sum_{k \in \mathbb Z_+} \Big(\sum_{ |I|=k}f_Iz^I \Big) t^k.\]

\begin{definition}\label{Spherically Balanced- definition}
    The Hilbert space $H^2(\beta)$ is said to be {\it spherically balanced} if the norm on $H^2(\beta)$
admits the {\it slice representation} $[\mu, H^2(\gamma)]$, that is, there exist a {\it Reinhardt measure} $\mu$
and a Hilbert
space $H^2(\gamma)$ of formal power series in one variable such that
\begin{eqnarray*}
     \|{f}\|^2_{H^2(\beta)} = \int_{\partial \mathbb
	B}\|{f_z}\|^2_{H^2(\gamma)}d\mu(z), \qquad \mbox{for ~all~} f \in
H^2(\beta)),
\end{eqnarray*}
where $\gamma =\{\gamma_k\}_{k \in \mathbb N}$ is given by the relation
$\beta_I=\gamma_{|I|}\|z^I\|_{L^2(\partial \mathbb B, \mu)}$ for all $I \in \mathbb{Z}_+^d.$
\end{definition}
By the Reinhardt measure, we mean a $\mathbb T^d$-invariant Borel probability measure supported in $\partial \mathbb B,$
where $\mathbb T^d=\{z \in \mathbb C^d:|z_1|=1,\ldots,|z_d|=1\}$ is the unit $d$-torus.

A complete classification of spherically balanced Hilbert spaces is provided in \cite[Theorem 1.10]{Ch-K}. For more details on spherically balanced Hilbert spaces, we refer to \cite{Ch-K} and \cite{Ku}. 
The following result describes a class of $d$-variable weighted shifts in which spherical contractivity ensures spherical dilation.
\begin{theorem} \label{thm-dila}
Let $H^2(\beta)$ be a spherically balanced Hilbert space and let $[\mu, H^2(\gamma)]$
be the slice representation for the norm on $H^2(\beta).$
Consider the $d$-tuple $M_z=(M_{z_1}, \ldots, M_{z_d})$ of multiplication by the
co-ordinate functions $z_1, \ldots, z_d$ on $H^2(\beta).$ 
If $M_z$ is a spherical contraction then it admits a  spherical dilation. 
\end{theorem}
\begin{proof} First observe that the norm on $H^2(\beta)^{\oplus n}$ admits the slice representation $[\mu, H^2(\gamma)^{\oplus n}],$ that is, \begin{eqnarray*}
     \|{f}\|^2_{H^2(\beta)^{\oplus n}} = \int_{\partial \mathbb
	B}\|{f_z}\|^2_{H^2(\gamma)^{\oplus n}}d\mu(z), \qquad \mbox{for ~all~}f \in
H^2(\beta)^{\oplus n}.
\end{eqnarray*}
Let \( p = (p_{i,j})_{1 \leq i,j \leq n} \) be a matrix with entries in \( \mathbb{C}[z_1, \dots, z_d] \). For each $z \in \mathbb C^d,$ the slice $p_z$ is the matrix \( p_z = (p_{i,j,z})_{1 \leq i,j \leq n} \), where the slice $p_{i,j,z}$ at $t\in \mathbb C$ is given as $p_{i,j,z}(t)=p_{i,j}(tz_1,\ldots,tz_d)$ for each $1 \leq i,j \leq n.$
Following the idea of the proof of \cite[Proposition 2.5]{Ku}, we get that 
\begin{eqnarray*} \|p(M_z)f\|^2_{H^2(\beta)^{\oplus n}} &=& \int_{\partial \mathbb
	B}\|p_z(\mathcal M_{t})f_z\|^2_{H^2(\gamma)^{\oplus n}}d\mu(z) \\
    &\leq &  \int_{\partial \mathbb B} \|p_z(\mathcal M_{t})\|^2_{\mathcal B(H^2(\gamma)^{\oplus n})}\|{f_z}\|^2_{H^2(\gamma)^{\oplus n}}d\mu(z).\end{eqnarray*} 
In the above expression, $\mathcal M_t$ denotes the multiplication operator by the co-ordinate function $t$ acting on the Hilbert space  $H^2(\gamma)$ of formal power series in one variable.  Since  $M_z$ is a spherical contraction, it follows from \cite[Remark 2.3(4)]{Ku} that $\mathcal M_t$ is a contraction. Consequently, we obtain
\begin{eqnarray*}
     \|p(M_z)f\|^2_{H^2(\beta)^{\oplus n}} &\leq &\int_{\partial \mathbb B}  \sup_{t \in \overline{\mathbb D}}\|p_z(t)\|^2_{\mathcal B(\mathbb C^n)}\|{f_z}\|^2_{H^2(\gamma)^{\oplus n}}d\mu(z) \\
     &\leq&  \sup_{w \in \overline{\mathbb B}}\|p(w)\|^2_{\mathcal B(\mathbb C^n)} \|f\|^2_{H^2(\beta)^{\oplus n}}.
\end{eqnarray*} 
This completes the proof.
\end{proof}
A $d$-variable weighted shift $T$ is said to be {\it balanced} if the norm on $H^2(\beta)$ admits the slice representation $[\mu, H^2(\gamma)].$ As an immediate corollary, we note the following result.
\begin{cor}\label{Balanced - Dilation}
If $T$ is a spherical contractive $d$-variable weighted shift which is balanced, then $T$ admits a spherical dilation. 
\end{cor}
 As evident from the above corollary, every spherical contractive, balanced $d$-variable weighted shift is row contractive (cf. \cite[Remark 2.6]{Ku}). On the other hand, a row contractive balanced $d$-variable weighted shift may not be spherical contractive, for example, the Drury-Arveson shift.
	
Example \ref{Counter-example: column} given below is due to Hartz, arose in a discussion at the OTOA conference in 2018 (ISI Bangalore), shows that the above result cannot be extended to an arbitrary spherically contractive $d$-variable weighted shifts.

\begin{example}\label{Counter-example: column}
Let us see the following diagram: 
	
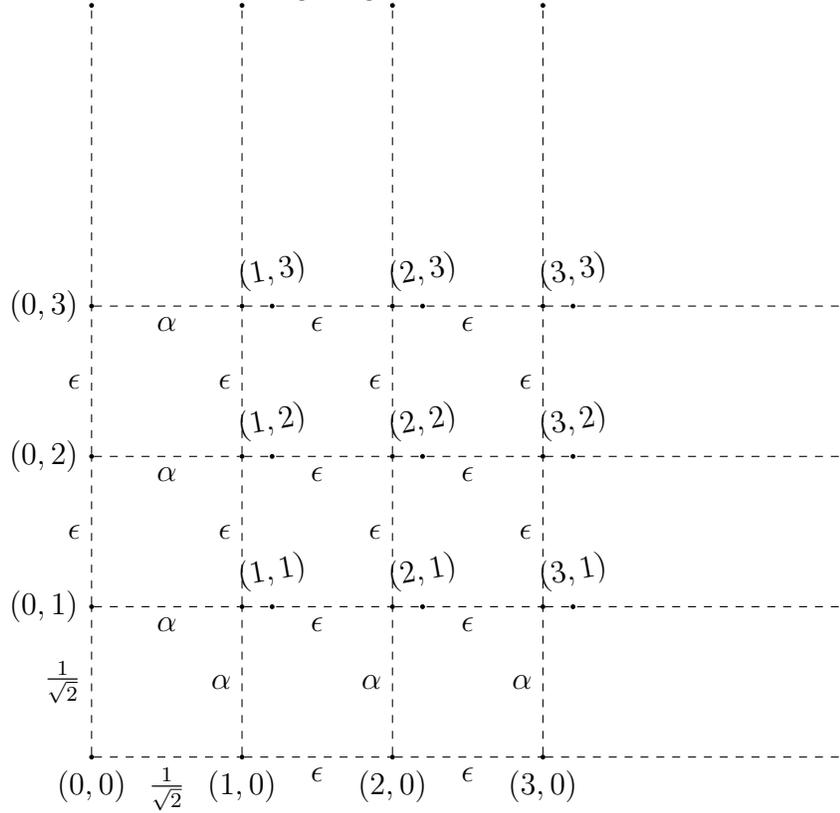
\tikzset{
vertex/.style={circle,draw, fill=black, inner
        sep=0.1pt},
edge/.style={->,> = latex'}
}

        \centering
\begin{tikzpicture}
[scale=2, vertices/.style={draw, fill=black, circle, inner
        sep=0.5pt}]
\node[vertices, label=below:{$(0,0)$}] (a) at (0,0) {};
\node[vertices, label=below:{$(1,0)$}] (b) at (1,0) {};
\node[vertices,label=left:{$(0,1)$}] (c) at (0,1) {};
\node[vertices, label=below:{$(2,0)$}] (d) at (2,0) {};
\node[vertices,label=left:{$(0,2)$}] (e) at (0,2) {};
\node[vertices, label=below:{$(3,0)$}] (f) at (3,0) {};
\node[vertices,label=left:{$(0,3)$}] (g) at (0,3) {};
\node[vertices ] (h) at (5,0) {};
\node[vertices] (i) at (0,5) {};
\node[vertices] (j) at (1,1) {};
\node[vertices,label=above:\rotatebox{10}{$(1,1)$}]  at (1.2,1) {};
\node[vertices] (k) at (2,1) {};
\node[vertices,label=above:\rotatebox{10}{$(2,1)$}]  at (2.2,1) {};
\node[vertices ] (l) at (3,1) {};
\node[vertices,label=above:\rotatebox{10}{$(3,1)$}]  at (3.2,1) {};
\node[vertices ] (m) at (5,1) {};
\node[vertices] (n) at (1,2) {};
\node[vertices,label=above:\rotatebox{10}{$(1,2)$}]  at (1.2,2) {};
\node[vertices] (o) at (1,3) {};
\node[vertices,label=above:\rotatebox{10}{$(1,3)$}]  at (1.2,3) {};
\node[vertices ] (p) at (1,5) {};
\node[vertices] (q) at (2,2) {};
\node[vertices,label=above:\rotatebox{10}{$(2,2)$}]  at (2.2,2) {};
\node[vertices] (r) at (3,2) {};
\node[vertices,label=above:\rotatebox{10}{$(3,2)$}]  at (3.2,2) {};
\node[vertices ] (s) at (5,2) {};
\node[vertices] (t) at (2,3) {};
\node[vertices,label=above:\rotatebox{10}{$(2,3)$}]  at (2.2,3) {};
\node[vertices ] (u) at (2,5) {};
\node[vertices] (v) at (3,3) {};
\node[vertices,label=above:\rotatebox{10}{$(3,3)$}]  at (3.2,3) {};
\node[vertices ] (w) at (5,3) {};
\node[vertices ] (x) at (3,5) {};

\draw[dashed] (a) -- (b) node[midway, below] {$\frac{1}{\sqrt{2}}$};
\draw[dashed] (a) -- (c) node[midway, left] {$\frac{1}{\sqrt{2}}$};
\draw[dashed] (b) -- (d) node[midway, below] {$\epsilon$};
\draw[dashed] (c) -- (e) node[midway, left] {$\epsilon$};
\draw[dashed] (d) -- (f) node[midway, below] {$\epsilon$};
\draw[dashed] (e) -- (g) node[midway, left] {$\epsilon$};
\draw[dashed] (f) -- (h) ;
\draw[dashed] (g) -- (i) ;
\draw[dashed] (b) -- (j) node[midway, left] {$\alpha$};
\draw[dashed] (d) -- (k) node[midway, left] {$\alpha$};
\draw[dashed] (f) -- (l) node[midway, left] {$\alpha$};
\draw[dashed] (c) -- (j) node[midway, below] {$\alpha$};
\draw[dashed] (e) -- (n) node[midway, below] {$\alpha$};
\draw[dashed] (g) -- (o) node[midway, below] {$\alpha$};
\draw[dashed] (j) -- (k) node[midway, below] {$\epsilon$};
\draw[dashed] (k) -- (l) node[midway, below] {$\epsilon$};
\draw[dashed] (n) -- (q) node[midway, below] {$\epsilon$};
\draw[dashed] (q) -- (r) node[midway, below] {$\epsilon$};
\draw[dashed] (o) -- (t) node[midway, below] {$\epsilon$};
\draw[dashed] (t) -- (v) node[midway, below] {$\epsilon$};
\draw[dashed] (l) -- (m) ;
\draw[dashed] (r) -- (s) ;
\draw[dashed] (v) -- (w) ;
\draw[dashed] (j) -- (n) node[midway, left] {$\epsilon$};
\draw[dashed] (n) -- (o) node[midway, left] {$\epsilon$};
\draw[dashed] (k) -- (q) node[midway, left] {$\epsilon$};
\draw[dashed] (q) -- (t) node[midway, left] {$\epsilon$};
\draw[dashed] (l) -- (r) node[midway, left] {$\epsilon$};
\draw[dashed] (r) -- (v) node[midway, left] {$\epsilon$};
\draw[dashed] (o) -- (p) ;
\draw[dashed] (t) -- (u) ;
\draw[dashed] (v) -- (x) ;
\end{tikzpicture}

Here $\alpha$ and $\epsilon$ are two positive numbers satisfying $\alpha^2+\epsilon^2=1$ and $\epsilon < \frac{1}{\sqrt{2}}$. \vspace{4mm}
\end{example}

Let $T=(T_1,T_2)$ be a two variable weighted shift whose weights are given in the diagram above, then
 \[
T_1 e_I=\left\{ \begin{array}{rl}
   \frac{1}{\sqrt{2}}e_{I+\varepsilon_1}, & \text{ if } I=(0,0)\\
   \alpha e_{I+\varepsilon_1}, & \text{ if } I=(0,k), \,\, k>0,\\
   \epsilon e_{I+\varepsilon_1}, &  \text{ otherwise}
                                   \end{array} \right. 
 \] and
 \[
T_2 e_I=\left\{ \begin{array}{rl}
   \frac{1}{\sqrt{2}}e_{I+\varepsilon_2}, & \text{ if } I=(0,0)\\
   \alpha e_{I+\varepsilon_2}, & \text{ if } I=(k,0), \,\, k>0,\\
   \epsilon e_{I+\varepsilon_2}, &  \text{ otherwise.}
                                   \end{array} \right. 
 \]
It is easy to verify that $(T_1,T_2)$  is a commuting pair of bounded linear operators on $\ell^2(\mathbb{Z}_+^2).$ If $I=(i_1,i_2),$ such that $i_1, i_2>0$, then $w^2_{I,1} + w^2_{I,2}=2 \epsilon^2 < 1.$ Hence $(T_1,T_2)$ is a spherical contraction. 
Consider the polynomial $p(z_1,z_2)= 2z_1 z_2.$ Since $\alpha > \frac{1}{\sqrt{2}},$ we have 
$$\|p(T_1,T_2)\|\geq \|2T_1T_2e_{(0,0)}\| = \|\sqrt{2} \alpha e_{(1,1)}\| > 1=\|p\|_{\infty}.$$
This shows that $(T_1,T_2)$ does not satisfy the von Neumann inequality. 
Note that $(T_1,T_2)$ is not a row contraction, as
\[\|T_1e_{(0,1)} + T_2e_{(1,0)}\|^2 = \|2\alpha e_{(1,1)}\|^2 = 4\alpha^2 > 2=\|e_{(0,1)}\|^2 + \|e_{(1,0)}\|^2.\]

\begin{example}\label{Counter-example: row}
Let us see the following diagram:

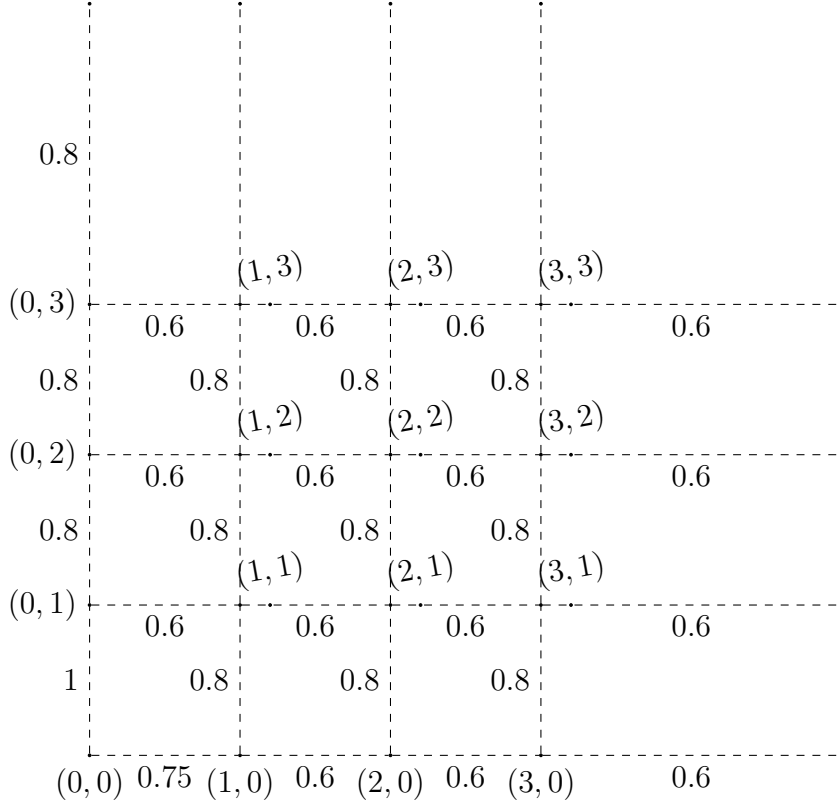
\tikzset{
  vertex/.style={circle,draw, fill=black, inner sep=0.1pt},
  edge/.style={->,> = latex'}
}

\begin{center}
\begin{tikzpicture}[scale=2, vertices/.style={draw, fill=black, circle, inner sep=0.5pt}]
    \node[vertex, label=below:{$(0,0)$}] (a) at (0,0) {};
    \node[vertex, label=below:{$(1,0)$}] (b) at (1,0) {};
    \node[vertex, label=left:{$(0,1)$}] (c) at (0,1) {};
    \node[vertex, label=below:{$(2,0)$}] (d) at (2,0) {};
    \node[vertex, label=left:{$(0,2)$}] (e) at (0,2) {};
    \node[vertex, label=below:{$(3,0)$}] (f) at (3,0) {};
    \node[vertex, label=left:{$(0,3)$}] (g) at (0,3) {};
    \node[vertex] (h) at (5,0) {};
    \node[vertex] (i) at (0,5) {};
    \node[vertex] (j) at (1,1) {};
    \node[vertex, label=above:\rotatebox{10}{$(1,1)$}] (j2) at (1.2,1) {};
    \node[vertex] (k) at (2,1) {};
    \node[vertex, label=above:\rotatebox{10}{$(2,1)$}] (k2) at (2.2,1) {};
    \node[vertex] (l) at (3,1) {};
    \node[vertex, label=above:\rotatebox{10}{$(3,1)$}] (l2) at (3.2,1) {};
    \node[vertex] (m) at (5,1) {};
    \node[vertex] (n) at (1,2) {};
    \node[vertex, label=above:\rotatebox{10}{$(1,2)$}] (n2) at (1.2,2) {};
    \node[vertex] (o) at (1,3) {};
    \node[vertex, label=above:\rotatebox{10}{$(1,3)$}] (o2) at (1.2,3) {};
    \node[vertex] (p) at (1,5) {};
    \node[vertex] (q) at (2,2) {};
    \node[vertex, label=above:\rotatebox{10}{$(2,2)$}] (q2) at (2.2,2) {};
    \node[vertex] (r) at (3,2) {};
    \node[vertex, label=above:\rotatebox{10}{$(3,2)$}] (r2) at (3.2,2) {};
    \node[vertex] (s) at (5,2) {};
    \node[vertex] (t) at (2,3) {};
    \node[vertex, label=above:\rotatebox{10}{$(2,3)$}] (t2) at (2.2,3) {};
    \node[vertex] (u) at (2,5) {};
    \node[vertex] (v) at (3,3) {};
    \node[vertex, label=above:\rotatebox{10}{$(3,3)$}] (v2) at (3.2,3) {};
    \node[vertex] (w) at (5,3) {};
    \node[vertex] (x) at (3,5) {};

    \draw[dashed] (a) -- (b) node[midway, below] {$0.75$};
    \draw[dashed] (a) -- (c) node[midway, left] {$1$};
    \draw[dashed] (b) -- (d) node[midway, below] {$0.6$};
    \draw[dashed] (c) -- (e) node[midway, left] {$0.8$};
    \draw[dashed] (d) -- (f) node[midway, below] {$0.6$};
    \draw[dashed] (e) -- (g) node[midway, left] {$0.8$};
    \draw[dashed] (f) -- (h) node[midway, below] {$0.6$};
    \draw[dashed] (g) -- (i) node[midway, left] {$0.8$};
    \draw[dashed] (b) -- (j) node[midway, left] {$0.8$};
    \draw[dashed] (d) -- (k) node[midway, left] {$0.8$};
    \draw[dashed] (f) -- (l) node[midway, left] {$0.8$};
    \draw[dashed] (c) -- (j) node[midway, below] {$0.6$};
    \draw[dashed] (e) -- (n) node[midway, below] {$0.6$};
    \draw[dashed] (g) -- (o) node[midway, below] {$0.6$};
    \draw[dashed] (j) -- (k) node[midway, below] {$0.6$};
    \draw[dashed] (k) -- (l) node[midway, below] {$0.6$};
    \draw[dashed] (n) -- (q) node[midway, below] {$0.6$};
    \draw[dashed] (q) -- (r) node[midway, below] {$0.6$};
    \draw[dashed] (o) -- (t) node[midway, below] {$0.6$};
    \draw[dashed] (t) -- (v) node[midway, below] {$0.6$};
    \draw[dashed] (l) -- (m) node[midway, below] {$0.6$};
    \draw[dashed] (r) -- (s) node[midway, below] {$0.6$};
    \draw[dashed] (v) -- (w) node[midway, below] {$0.6$};
    \draw[dashed] (j) -- (n) node[midway, left] {$0.8$};
    \draw[dashed] (n) -- (o) node[midway, left] {$0.8$};
    \draw[dashed] (k) -- (q) node[midway, left] {$0.8$};
    \draw[dashed] (q) -- (t) node[midway, left] {$0.8$};
    \draw[dashed] (l) -- (r) node[midway, left] {$0.8$};
    \draw[dashed] (r) -- (v) node[midway, left] {$0.8$};
    \draw[dashed] (o) -- (p);
    \draw[dashed] (t) -- (u);
    \draw[dashed] (v) -- (x);
\end{tikzpicture}
\end{center}

\end{example}
Let $T=(T_1,T_2)$ be a two variable weighted shift with weights given as in the preceding diagram, then
 \[
T_1 e_I=\left\{ \begin{array}{rl}
   \frac{3}{{4}}e_{I+\varepsilon_1}, & \text{ if } I=(0,0)\\
   
   0.6 e_{I+\varepsilon_1}, &  \text{ otherwise}
                                   \end{array} \right. 
 \] and
 \[
T_2 e_I=\left\{ \begin{array}{rl}
   1e_{I+\varepsilon_2}, & \text{ if } I=(0,0)\\
   
   0.8  e_{I+\varepsilon_2}, &  \text{ otherwise}.
                                   \end{array} \right. 
 \]
A routine verification shows that $(T_1,T_2)$ is a row contractive commuting tuple of bounded linear operators on $\ell^2(\mathbb Z_+^2).$
We contend that $(T_1,T_2)$ does not satisfy the von Neumann's inequality. If $p(z_1,z_2)= 2z_1 z_2,$ then
\[ \|p(T_1, T_2)\|\geq \|2T_1T_2e_{(0,0)}\| = \|1.2 e_{(1,1)}\| > 1. \] 
On the other hand, $\|p\|_{ \infty}=1$. Thus $(T_1,T_2)$ does not satisfy the von Neumann inequality. Note that $(T_1,T_2)$ is not a spherical contraction, because if we choose $h=e_{0,0},$ then $\|T_1e_{(0,0)}\|^2 +\| T_2e_{(0,0)}\|^2 = (0.75)^2+1 = 1.5625 > 1.$

\section{Proof of Theorem \ref{Main Theorem}}
In this section, we proceed to prove  Theorem \ref{Main Theorem}.
The proof present here is motivated by the techniques in \cite{Hz}.
Fix \( N \in \mathbb{Z}_+ \) and define
\[
\mathcal{I} = \{ (I,j) \in \mathbb{Z}_+^d \times \{1, \dots, d\} : |I| \leq N \}.
\]
Let \( X \) denote the closure of the set of all \( ({w}_{I,j})_{(I,j) \in \mathcal{I}} \) satisfying commuting relations \eqref{commutation relation} such that for all $I\in\mathbb{Z}_+^d$ such that
\[
0 <  |\mathrm{w}_{I,j}| \leq m_j  \quad  \text{and} \sum_{j=1}^d m_j^2 \leq1
\]
for all \( (I,j) \in \mathcal{I} \). We may regard \( X \) as a compact subset of \( \mathbb{C}^{|\mathcal{I}|} \). Define
\[
X_0 = \left\{ (\mathrm{w}_{I,j}) \in X : 
 |\mathrm{w}_{I,j}| = m_j,\quad
\sum_{j=1}^d m_j^2 = 1 
\right\}
\]
For any compact set $S \subset \mathbb{C}^N,$ let $\partial_0 S$ denote the Shilov boundary of the algebra of all analytic functions on $S$. That is, $\partial_0 S$ is the smallest compact subset $K \subset S$ such that 
\[
\sup\{|f(z)| : z \in S\} = \sup\{|f(z)| : z \in K\}
\]
for every analytic function $f$ on $S$.

\begin{lemma}\label{Shilov}
The Shilov boundary of \( X \) is contained in $X_0.$
\end{lemma}
The proofs of Lemmas \ref{Shilov} and \ref{Sperical unitary dilation for truncated weighted shift} follow the same line of argument; accordingly, the proof of Lemma \ref{Shilov} is incorporated into that of Lemma \ref{Sperical unitary dilation for truncated weighted shift}. The proof is based on the following idea.

Let \( \mathrm{w} = (w_{I,j}) \in X \setminus X_0 \) with \( w_{I,j} \neq 0 \) for all \( (I, j) \in \mathcal{I} \), and let
\( f : X \to \mathbb{C} \) be a function which extends to be analytic in a neighbourhood of \( X \).
We will show that there exists a point \(\tilde{ \mathrm{w}} = (\tilde{w}_{I,j}) \in X \) with \( \tilde{w}_{I,j} \neq 0 \) for all \( (I, j) \in \mathcal{I} \) such that
\[
|f(\mathrm{w})| \leq |f(\tilde{\mathrm{w}})|
\]
and such that

\[\Big\{
(I, j) \in \mathcal{I} :
 |w_{I,j}| = m_j,\ 
\sum_{j=1}^d  m_j^2 = 1
\Big\} 
\subsetneq 
\Big\{
(I, j) \in \mathcal{I} :
 |\tilde{w}_{I,j}| = m_j,\
\sum_{j=1}^d m_j^2 = 1
\Big\}\]

Once this has been accomplished, iterating this process finitely many times
yields a point \( \mathrm{v} \in X_0 \) such that \( |f(\mathrm{w})| \leq |f(\mathrm{v})| \).
Consequently, \( X_0 \) is a boundary for the algebra of all analytic functions on \( X \), so
\[
\partial_0 X \subset X_0.
\]

\begin{remark}
    It is worth mentioning that, if $d=2,$ then the set $X_0$ is precisely the set
    \[\big\{(w_{I,j}): |w_{I,j}|=m_j, m_1^2+m_2^2=1, m_1>0 \mbox{ }m_2>0\big\}.\]
To see this, let \( \tilde{w}_{0,1} = m_1 \), \( \tilde{w}_{0,2} = m_2 \), \( \tilde{w}_{\epsilon_2,1} = x \), and \( \tilde{w}_{\epsilon_1,2} = y \). Using commutativity, we get \( m_2 \cdot x = m_1 \cdot y \), and with \( m_1^2 + m_2^2 = 1 \) and \(x^2+y^2 =1\), we get:
\[
\frac{m_2^2}{m_1^2} = \frac{y^2}{x^2} \Rightarrow \frac{m_2^2 + m_1^2}{m_1^2} = \frac{x^2 + y^2}{x^2} \Rightarrow x = m_1, \, y = m_2.\]    
Similarly, we get $w_{I,1}=m_1$ and $w_{I,2}=m_2$ for each $I\in\mathbb Z_+^2.$    
\end{remark}
Now, we continue with the development of the proof of  Theorem \ref{Main Theorem}. The proof of the following proposition is same (after required modification) as in \cite[Proposition 2.1]{Hz}. We present the proof for the sake of completeness.
\begin{proposition}\label{Arveson}
    Let \(X \subset\mathbb{C}^N  \) be compact, and suppose that \(T: X\to \mathcal{B}(\mathcal{H)}^d\) is an analytic function such that \(T(z)\) is a d-tuple of commuting bounded operators  for all \(z\in X\). Then the following statements are true: 
    \begin{enumerate}
        \item If the tuple \(T(z)\) satisfies  von Neumann's inequality over $\mathbb B_d$ for all \(z\in \partial_0 X\) then $T(z)$ satisfies von Neumann's inequality over $\mathbb B_d$ for all $z\in  X$.
        \item If the tuple \(T(z)\) dilates to a tuple of commuting spherical unitaries for all \(z\in \partial_0 X\) then $T(z)$ dilates to a tuple of commuting spherical unitaries for all $z\in X$.
    \end{enumerate}
\end{proposition}

\begin{proof}
Let \( p = (p_{i,j})_{1 \leq i,j \leq n} \) be an \( n \times n \) matrix of polynomials in \( \mathbb{C}[z_1, \dots, z_d] \), and suppose that the inequality $\|p(T(z))\|_{B(H^{\oplus n})} \leq \|p\|_\infty$
holds for all \( z \in \partial_0 X \), where \( \|p\|_\infty = \sup \{ \|p(w)\|_{M_n} : w \in\overline{\mathbb B}_d \} \). 
Given \( f, g \in H^{\oplus n} \) of norm 1, observe that the scalar-valued function
\[
X \to \mathbb{C}, \quad z \mapsto \langle p(T(z))f, g \rangle,
\]
is analytic. By assumption, this function is bounded by \( \|p\|_\infty \) on \( \partial_0 X \), and hence on \( X \) by the definition of \( \partial_0 X \). Consequently, the inequality
$\|p(T(z))\|_{B(H^{\oplus n})} \leq \|p\|_\infty$
holds for all \( z \in X \). Part (1) now follows by taking \( n = 1 \) above.

To prove the second part, let's assume that $T(z)$  dilates to a $d$-tuple of commuting spherical unitaries  for all $z\in \partial _0X$. Then $T(z)$  satisfies matrix-valued von Neumann's inequality over $\mathbb B_d$ for all $z \in \partial_0 X.$ By part (1), it also satisfies matrix-valued von Neumann's inequality for all $z\in X.$ Therefore, for every $z\in X,$ it follows that the map \( p\mapsto p(T(z))\) from $M_n(p[z_1,\dots,z_d])\to M_n(B(H))$ is completely contractive which is also unital homomorphism. Therefore by Arveson's extension theorem 
 \cite[Corollary 7.7]{VP}, 
 we can extend this homomorphism to $*$-homomorphism on the space of continuous functions on $\partial \mathbb B_d.$ Hence, for every $z\in X,$ we get a spherical unitary dilation for $T(z)$. 
 \end{proof}
The following result also follows from Corollary \ref{Balanced - Dilation} as it is a balanced weighted shift. However, below we provide a direct proof.
\begin{proposition}\label{Dilation for X_0}
    Let $\boldsymbol T=(T_1, \ldots, T_d)$ be a doubly contractive $d$-variable weighted shift with weights $\{w_{I,j}: I\in\mathbb{Z}_+^d, j\in\{1,\ldots, d\}\}$ satisfying $w_{I,j}=m_j$ for some real numbers $m_1,\ldots,m_d,$ and $\sum_{j=1}^{d} m_j^2 \leq 1.$
Then $\boldsymbol T$ admits a spherical dilation. In particular, $\boldsymbol T$ satisfies von Neumann’s inequality.
\end{proposition}
\begin{proof}
    Let $m_k = \min\{m_j : j = 1, 2, \ldots, d\}$. If the minimum occurs for multiple values, choose one arbitrarily; . Define
\[
\tilde{m}_k = \sqrt{1 - \sum_{j \neq k} m_j^2}, \quad \text{and} \quad r = \frac{\tilde{m}_k}{m_k}\geq 1,
\]
\textbf{Case 1.} If $r=1$ in that case $\sum m_j^2 =1$.Thus $\boldsymbol T$ is spherical isometry then by \cite{A.A} we obtain a spherical unitary dilation, and hence von Neumann's inequality holds.\\
\textbf{Case 2.}If $r>1$ .
We consider $w_{I,j} = m_j$ for all $I$ and for all $j = 1, 2, \ldots, d$. Let $\mathbb{D}_r(0) \subset \mathbb{C}$ denote the closed disc of radius $r$ centered at $0$. For $(I,j) \in \mathbb{Z}_+^d \times \{1, 2, \ldots, d\}$, define
\[
\hat{w}_{I,j}(t) = 
\begin{cases}
t \cdot w_{I,j} & \text{if } w_{I,k} = m_k, \\
w_{I,j} & \text{otherwise},
\end{cases}
\]
and let $\hat{w}(t) = (\hat{w}_{I,j}(t))_{(I,j)\in \mathbb{Z}_+^d \times \{1,\ldots,d\}}$.
We finish the proof by showing that $\hat{w}(t) \in X$ for every $t \in \mathbb{D}_r(0)$. Indeed, it then follows from the maximum modulus principle that there exists $t_0 \in \partial \mathbb{D}_r(0)$ and  for every holomorphic $f:X \to \mathbb{C}$ such that
\[
|f(w)| = |f(\hat{w}(1))| \leq |f(\hat{w}(t_0))|.
\]
Setting $\hat{w} = \hat{w}(t_0)$, we obtain a point $\hat{w} \in X_0$ such that all weights are $\tilde{m}_j$ in the respective directions. This transformation preserves commutativity, as well as the row and column contractivity conditions.
By \cite{A.A}, we then obtain a spherical unitary dilation, and hence von Neumann’s inequality holds. 
\end{proof}

Thus, in order to prove Theorem \ref{Main Theorem}, it suffices to prove the following lemma. 
\begin{lemma}\label{Sperical unitary dilation for truncated weighted shift}
Suppose $m_1,\ldots, m_d$ are positive real numbers such that \(\sum_{j=1}^d m_j^2 \leq1\).
Let $\boldsymbol T=(T_1,\ldots, T_d)$ be a \( d \)-variable doubly contractive truncated weighted shift with weights $(w_{I,j})$ such that \( w_{I,j}\in (0, m_j] \) for all $I\in\mathbb{Z}_+^d$ and $j\in\{1,\ldots, d\}.$ Then $\boldsymbol T$  dilates to a \( d \)-tuple of commuting spherical unitaries.
\end{lemma}

The argument of sufficiency of 
Lemma \ref{Sperical unitary dilation for truncated weighted shift}  to prove Theorem \ref{Main Theorem} follows from \cite[Lemma 4.1]{Hz}.
To prove Lemma \ref{Sperical unitary dilation for truncated weighted shift}, we need a definition and a few observations.

\begin{definition}[Scalability]
    A multi-index \( I = (i_1, i_2, \dots, i_d) \in \mathbb{Z}_+^d \) is said to be scalable in direction \( j \) (or \((I, j)\) is scalable) if the following conditions are satisfied: \( I \) is good, and \( I+\epsilon_j \) is bad.
A multi-index \( I \)  is said to be good if \[\|T^I e_{(0,\dots ,0)}\|=m_1^{i_1}m_2^{i_2}\cdots m_d^{i_d}\] that means the big rectangular box with diagonal endpoints \( (0, 0, \dots, 0) \) and \( I \) has the property that all the $j$-th direction weights are $m_j,$ for each $j=1,\ldots, d.$ 
Otherwise, we call \( I \) to be a bad index. 
\end{definition}

The following observations are immediate:
\begin{itemize}
    \item[(a)] If \( I \) is good and if \( J \leq I \), then \( J \) is good.
    \item[(b)] If \( (I, j) \in \mathcal{I} \) with \( |w_{I,j}| < m_j \), then \( I + \epsilon_j \) is bad.
    \item[(c)] Suppose that \( |I| \leq N \). If \( I \) is good and \( I + \epsilon_j \) is bad, then \( |w_{I,j}| < m_j \).
\end{itemize}

\begin{proof}[Proof of Lemma \ref{Sperical unitary dilation for truncated weighted shift}:]
Define \[
r := \left( 
\max \left\{ 
\frac{|w_{I,j}|}{m_j} :  (I, j) \text{ is scalable} 
\right\}
\right)^{-1}.
\]
Then \( r> 1\).  Let \( D_r(0) \subset \mathbb{C}\) be the  closed unit disc of radius \(r\)  around 0 . For \(t\in D_r(0)\) and \( 
(I,j)\in \mathcal{I} \) , define 
\[
\hat{w}_{I,j}(t) = 
\begin{cases}
t w_{I,j} & \text{if } (I,j) \text{ is scalable }, \\
w_{I,j} & \text{otherwise},
\end{cases}
\]
and let \( \hat{w}(t) = (\hat{w}_{I,j}(t))_{(I,j) \in I} \). Then it follows from the maximum modulus principle that there exists \( t_0 \in \partial D_r(0) \) with
\[
|f(w)| = |f(\hat{w}(1))| \leq |f(\hat{w}(t_0))|,\] so setting \( \hat{w} = \hat{w}(t_0) \), iterating this process finitely many times yields a point \( v \in {X_0} \) such that all  weights are  \(m_j\) in respective direction  such that \( |f(w)| \leq |f(v)| \) and also satisfying commutativity without changing row and column contraction. So we need to show that \( \hat{w}(t) \) is a commuting family, that is
\[
\hat{w}_{I,j}(t) \hat{w}_{I+\varepsilon_j,k}(t) = \hat{w}_{I,k}(t) \hat{w}_{I+\varepsilon_k,j}(t)
\]
for all \( t \in D_r(0) \) and all multi-indices \( I \) with \( |I| \leq N-1 \) and \( 1 \leq j, k \leq d \).
Let \( I \) be such a multi-index. If \( I \) is bad, it follows from (a) that \( I + \varepsilon_j \) and \( I + \varepsilon_k \) are bad as well, and hence no pairs in \( I \) which appear in the above equation are scalable. If \( I \) and \( I + \varepsilon_j + \varepsilon_k \) are good, then it follows again from (a) that no pairs in the equation are scalable. Thus, it remains to consider the case where \( I \) is good and \( I + \varepsilon_j + \varepsilon_k \) is bad. In this case, exactly one of \( (I,j) \) and \( (I + \varepsilon_j, k) \) is scalable, depending on whether \( I + \varepsilon_j \) is good or bad. Similarly, exactly one of \( (I,k) \) and \( (I + \varepsilon_k, j) \) is scalable. Thus,
\[\hat{w}_{I,j}(t) \hat{w}_{I+\varepsilon_j,k}(t) = t w_{I,j} w_{I+\varepsilon_j,k} = t w_{I,k} w_{I+\varepsilon_k,j} = \hat{w}_{I,k}(t) \hat{w}_{I+\varepsilon_k,j}(t),
\]
as stated. 
Now the proof of Lemma \ref{Shilov} is complete.

Note that Proposition \ref{Dilation for X_0} shows that any weighted shift in $X_0$ has spherical dilation. In particular, using Lemma \ref{Shilov}, we have spherical dilation for any weighted shift in $\partial_0 X.$ An application of Proposition \ref{Arveson} now completes the proof of lemma.
\end{proof}
The following example illustrates that the methods used in the proof of Theorem \ref{Main Theorem} are insufficient to determine whether a doubly contractive weighted shift admits a spherical dilation.
\begin{example}
Consider a two variable weighted shift whose weights are given in the diagram below:
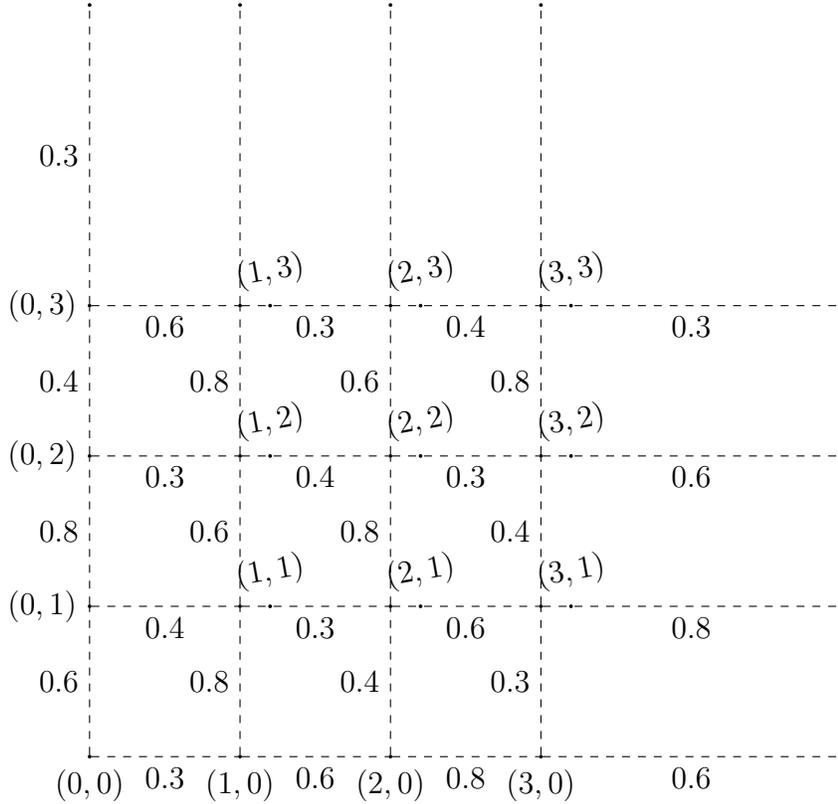
\tikzset{
  vertex/.style={circle,draw, fill=black, inner sep=0.1pt},
  edge/.style={->,> = latex'}
}

\begin{center}
\begin{tikzpicture}[scale=2, vertices/.style={draw, fill=black, circle, inner sep=0.5pt}]
    \node[vertex, label=below:{$(0,0)$}] (a) at (0,0) {};
    \node[vertex, label=below:{$(1,0)$}] (b) at (1,0) {};
    \node[vertex, label=left:{$(0,1)$}] (c) at (0,1) {};
    \node[vertex, label=below:{$(2,0)$}] (d) at (2,0) {};
    \node[vertex, label=left:{$(0,2)$}] (e) at (0,2) {};
    \node[vertex, label=below:{$(3,0)$}] (f) at (3,0) {};
    \node[vertex, label=left:{$(0,3)$}] (g) at (0,3) {};
    \node[vertex] (h) at (5,0) {};
    \node[vertex] (i) at (0,5) {};
    \node[vertex] (j) at (1,1) {};
    \node[vertex, label=above:\rotatebox{10}{$(1,1)$}] (j2) at (1.2,1) {};
    \node[vertex] (k) at (2,1) {};
    \node[vertex, label=above:\rotatebox{10}{$(2,1)$}] (k2) at (2.2,1) {};
    \node[vertex] (l) at (3,1) {};
    \node[vertex, label=above:\rotatebox{10}{$(3,1)$}] (l2) at (3.2,1) {};
    \node[vertex] (m) at (5,1) {};
    \node[vertex] (n) at (1,2) {};
    \node[vertex, label=above:\rotatebox{10}{$(1,2)$}] (n2) at (1.2,2) {};
    \node[vertex] (o) at (1,3) {};
    \node[vertex, label=above:\rotatebox{10}{$(1,3)$}] (o2) at (1.2,3) {};
    \node[vertex] (p) at (1,5) {};
    \node[vertex] (q) at (2,2) {};
    \node[vertex, label=above:\rotatebox{10}{$(2,2)$}] (q2) at (2.2,2) {};
    \node[vertex] (r) at (3,2) {};
    \node[vertex, label=above:\rotatebox{10}{$(3,2)$}] (r2) at (3.2,2) {};
    \node[vertex] (s) at (5,2) {};
    \node[vertex] (t) at (2,3) {};
    \node[vertex, label=above:\rotatebox{10}{$(2,3)$}] (t2) at (2.2,3) {};
    \node[vertex] (u) at (2,5) {};
    \node[vertex] (v) at (3,3) {};
    \node[vertex, label=above:\rotatebox{10}{$(3,3)$}] (v2) at (3.2,3) {};
    \node[vertex] (w) at (5,3) {};
    \node[vertex] (x) at (3,5) {};

    \draw[dashed] (a) -- (b) node[midway, below] {$0.3$};
    \draw[dashed] (a) -- (c) node[midway, left] {$0.6$};
    \draw[dashed] (b) -- (d) node[midway, below] {$0.6$};
    \draw[dashed] (c) -- (e) node[midway, left] {$0.8$};
    \draw[dashed] (d) -- (f) node[midway, below] {$0.8$};
    \draw[dashed] (e) -- (g) node[midway, left] {$0.4$};
    \draw[dashed] (f) -- (h) node[midway, below] {$0.6$};
    \draw[dashed] (g) -- (i) node[midway, left] {$0.3$};
    \draw[dashed] (b) -- (j) node[midway, left] {$0.8$};
    \draw[dashed] (d) -- (k) node[midway, left] {$0.4$};
    \draw[dashed] (f) -- (l) node[midway, left] {$0.3$};
    \draw[dashed] (c) -- (j) node[midway, below] {$0.4$};
    \draw[dashed] (e) -- (n) node[midway, below] {$0.3$};
    \draw[dashed] (g) -- (o) node[midway, below] {$0.6$};
    \draw[dashed] (j) -- (k) node[midway, below] {$0.3$};
    \draw[dashed] (k) -- (l) node[midway, below] {$0.6$};
    \draw[dashed] (n) -- (q) node[midway, below] {$0.4$};
    \draw[dashed] (q) -- (r) node[midway, below] {$0.3$};
    \draw[dashed] (o) -- (t) node[midway, below] {$0.3$};
    \draw[dashed] (t) -- (v) node[midway, below] {$0.4$};
    \draw[dashed] (l) -- (m) node[midway, below] {$0.8$};
    \draw[dashed] (r) -- (s) node[midway, below] {$0.6$};
    \draw[dashed] (v) -- (w) node[midway, below] {$0.3$};
    \draw[dashed] (j) -- (n) node[midway, left] {$0.6$};
    \draw[dashed] (n) -- (o) node[midway, left] {$0.8$};
    \draw[dashed] (k) -- (q) node[midway, left] {$0.8$};
    \draw[dashed] (q) -- (t) node[midway, left] {$0.6$};
    \draw[dashed] (l) -- (r) node[midway, left] {$0.4$};
    \draw[dashed] (r) -- (v) node[midway, left] {$0.8$};
    \draw[dashed] (o) -- (p);
    \draw[dashed] (t) -- (u);
    \draw[dashed] (v) -- (x);
\end{tikzpicture}
\end{center}
\end{example}
This is an example of a doubly contractive two variable weighted shift whose weights are bounded above by $m_1=m_2=0.8$ satisfying $m_1^2 +m_2^2 =1.28>1.$ If we apply the techniques used in the proof of Theorem \ref{Main Theorem}, then we end up getting a weighted shift which is neither a row contraction nor a spherical contraction.
On the other hand, we do not know whether the weighted shift in the above diagram admits a spherical dilation.

We conclude this article with following positive result which shows that  any doubly contractive operator tuple satisfies von Neumann inequality for all homogeneous polynomial of degree $2.$
\begin{proposition}\label{von Neumann's inequality for doubly contractive operators}
 Let \( \boldsymbol T = (T_1, \ldots, T_d) \) be a doubly contractive operator tuple on a Hilbert space $\mathcal{H}$ and $p\in \mathbb{C}[z_1,\ldots ,z_d]$  be a homogeneous polynomial of degree $2$ in $d$-variables then $\|p(T)\|\leq \|p\|_{\infty}$. 
 \end{proposition} 
 \begin{proof}
 Let $p(z_1,\ldots, z_d)=\sum_{i,j=1}^d a_{ij}\,z_i z_j$ be a homogeneous polynomial of degree $2.$ Without loss of generality, we can assume that the matrix $A=(a_{ij})_{1\le i,j\le d}$ is  symmetric. 
Note that
\[
p(T_1,\ldots,T_d)
 \;=\; \sum_{i,j=1}^d a_{ij}\,T_i T_j
 \;=\; (T_1,\ldots, T_d)\;
     (A \otimes I)\;
     \begin{pmatrix}
       T_1 \\[4pt] T_2 \\[4pt] \vdots \\[4pt] T_d
     \end{pmatrix}.
\]
Since $T$ is doubly contractive,  we get
$$\|p(T_1,T_2,\ldots,T_d)\| \leq \|A\otimes I\| = \|A\|_2.$$ 
Next, we claim that $\|A\|=\|p(z)\|_{\infty}$. This will complete the proof of the proposition. To this end, observe that 
by the Takagi factorization for complex symmetric matrices, there exists a unitary matrix \(U\in\mathbb{C}^{d\times d}\) and a nonnegative real diagonal matrix \(\Sigma=\operatorname{diag}(\sigma_1,\dots,\sigma_d)\) with \(\sigma_1\ge\sigma_2\ge\cdots\ge0\) such that
\[
A = U\Sigma U^T.
\]
The diagonal entries \(\sigma_i\) are the Takagi singular values of \(A\), see \cite{T.K} (equivalently the largest singular value / spectral norm \(\|A\|_2\)). In particular, \(\sigma_1=\|A\|_2\).
Note that 
\begin{align*}
p(z) &= z^T A z = z^T(U\Sigma U^T)z = (U^T z)^T \Sigma (U^T z).
\end{align*}
Therefore, we get $\|p(z)\|_{\infty}=\|\Sigma\|=\sigma_1=\|A\|.$ Thus the claim stands verified. 
  \end{proof}

\noindent\textbf{Acknowledgment:} The work of the third named author was supported by the Anusandhan National Research Foundation (ANRF) through IRG research grant (Ref. No. ANRF/IRG/2024/000432/MS). The authors are grateful to Dr. Hartz for his fruitful discussions at the OTOA conference in 2018 (ISI Bangalore). 
We also express our gratitude to Prof. Cherian Varughese and Prof. Gadadhar Misra for organizing the informal conference `Operator Analysis - A Renaissance' in 2024 at Gift City Club, Gandhinagar,  during which this problem was discussed and several participants provided helpful comments.

\bibliographystyle{amsplain}
	
\end{document}